\documentclass[12pt,a4paper]{article}
\usepackage{amsfonts}
\usepackage{epsfig}
\usepackage{latexsym}
\usepackage{amsthm}
\usepackage{amsmath}
\usepackage{amssymb}
\usepackage{array}
\usepackage{enumerate}
\usepackage{tikz}
\usetikzlibrary{shadings,patterns}
\usepackage{graphicx,graphics}
\newtheorem{theorem}{Theorem}[section]
\newtheorem{lemma}[theorem]{Lemma}
\newtheorem{proposition}[theorem]{Proposition}

\newtheorem{corollary}[theorem]{Corollary}
\newtheorem{definition}[theorem]{Definition}
\newtheorem{remark}[theorem]{Remark}

\begin{document}
\date{}
\title{\bf Spectrum of conjugacy and order super commuting graphs of some finite groups}
\author{Sandeep Dalal \and Sanjay Mukherjee\footnote{Supported by CSIR Grant No. 09/1248(0004)/2019-EMR-I, MHRD, Government of India} \and Kamal Lochan Patra\footnote{Partially supported by Project No. MTR/2022/000424 of SERB, Government of India}}
\maketitle
\begin{abstract}
Let $\Gamma$ be a simple finite graph with vertex set $V(\Gamma)$ and edge set $E(\Gamma)$. Let $\mathcal{R}$ be an equivalence relation on $V(\Gamma)$. The $\mathcal{R}$-super $\Gamma$ graph $\Gamma^{\mathcal{R}}$ is a simple graph with vertex set $V(\Gamma)$ and two distinct vertices are adjacent if either they are in the same $\mathcal{R}$-equivalence class or there are elements in their respective $\mathcal{R}$-equivalence classes that are adjacent in the original graph $\Gamma$. We first show that $\Gamma^{\mathcal{R}}$ is a generalized join of some complete graphs and using this we obtain the adjacency and Laplacian spectrum of conjugacy and order super commuting graphs of dihedral group $D_{2n}\; (n\geq 3)$, generalized quaternion group $Q_{4m} \;(m\geq 2)$ and the nonabelian group $\mathbb Z_p \rtimes \mathbb Z_q$ of order $pq$, where $p$ and $q$ are distinct primes with $q|p-1$. 
\end{abstract}


\section{Introduction}
Throughout this paper, graphs are finite, simple and undirected. Let $\Gamma$ be a graph with vertex set $V(\Gamma)$ and edge set $E(\Gamma)$. For $u,v\in V(\Gamma)$, we say $u$ is \emph{adjacent} to $v$ if there is an edge between $u$ and $v$, and we denote it by $u \sim v$ (or $uv \in E(\Gamma)$). By degree of a vertex $v$ in $\Gamma$, we mean the number of edges incident with $v$ and we denote it as deg$(v)$. A vertex $v$ is called a pendant vertex if deg$(v)=1$. A vertex $v$ is called a dominant vertex of $\Gamma$ if it is adjacent with all other vertices of $\Gamma$. We denote the complete graph on $n$ vertices by $K_n$ and the star graph on $m+1$ vertices by $K_{1,m}$. We refer to \cite{We} for other unexplained graph theoretic terminologies used in this paper.

For two graphs $\Gamma_1$ and $\Gamma_2$ with disjoint vertex sets, the \emph{join graph} $\Gamma_1 \vee \Gamma_2$ of $\Gamma_1$ and $\Gamma_2$  is the graph obtained from the union of $\Gamma_1$ and $\Gamma_2$ by adding new edges from each vertex of $\Gamma_1$ to every vertex of $\Gamma_2.$ The following is a generalization of the definition of join graph (which is called generalized composition graph in \cite{Sc}):

\begin{definition}
Let  $\Im$ be a graph on $k$ vertices with $V(\Im) = \{v_1,v_2, \ldots, v_k\}$ and let $\Gamma_1, \Gamma_2, \ldots, \Gamma_k$ be pairwise vertex disjoint graphs. The $\Im$-generalized join graph $\Im [\Gamma_1, \Gamma_2, \ldots, \Gamma_k]$ of $\Gamma_1, \Gamma_2, \ldots, \Gamma_k$ is the graph 
formed by replacing each vertex $v_i$ of $\Im$ by the graph $\Gamma_i$ and then joining each vertex of $\Gamma_i$ to every vertex of $\Gamma_j$ whenever $v_i \sim v_j$ in $\Im$.
\end{definition}
Note that $K_2[\Gamma_1,\Gamma_2]$ coincides with the usual join $\Gamma_1 \vee \Gamma_2$ of $\Gamma_1$ and $\Gamma_2$. The following result is useful and tells about the connectedness of $\Im [\Gamma_1, \Gamma_2, \ldots, \Gamma_k]$.

\begin{lemma} [\cite{CHA}, Lemma 2.2]\label{Gj}
Let  $\Im$ be a graph on $k$ vertices with $V(\Im) = \{v_1,v_2, \ldots, v_k\}$ and let $\Gamma_1,\Gamma_2, \ldots, \Gamma_k$ be pairwise vertex disjoint graphs. If the $\Im$-generalized join graph $\Im [\Gamma_1, \Gamma_2, \ldots, \Gamma_k]$ is connected then $\Im$ is connected. Conversely, if $k\geq 2$ and $\Im$ is connected then $\Im [\Gamma_1, \Gamma_2, \ldots, \Gamma_k]$ is connected.
\end{lemma}

Graphs defined on groups have a long history. Many graphs are defined with vertex set being a group $G$ and the edge set reflecting the  structure of $G$ in some way, for example, Cayley graph, commuting graph, power graph, prime graph, intersection graph etc. For more on different graphs defined on groups, we refer to the survey paper \cite{Ca}. Add to this study, Arunkumar \emph{et al.} in \cite{AC}, introduced the notion of super graphs on groups. Let $G$ be a finite group and let $\mathcal{R}$ be an equivalence relation defined on $G$. For $g\in G$, let $[g]_{\mathcal{R}}$ be the $\mathcal{R}$-equivalence class of $g$ in $G$. Let $\Gamma$ be a graph with $V(\Gamma)=G$. The $\mathcal{R}$-super $\Gamma$ graph is the simple graph with vertex set $G$ and two vertices $g$ and $h$ are adjacent if and only if there exists $g'\in [g]_{\mathcal{R}}$ and $h'\in [h]_{\mathcal{R}}$ such that $g'$ and $h'$ are adjacent in the graph $\Gamma$. It is also assumed that the subgraph induced by the vertices of  $[g]_{\mathcal{R}}$ in the $\mathcal{R}$-super $\Gamma$ graph is complete and the reason has been discussed in \cite{AC} (see Section 1.2). In this article we have extended this notion of super graph on groups to any simple graph $\Gamma$ over an equivalence relation on $V(\Gamma)$.

The paper is organized as follows: In Section $2$, some known  results related to different matrices associated with graphs are discussed. In Section $3$, the $\mathcal{R}$-super graph of a graph is defined and its adjacency and Laplacian characteristic polynomials are studied. In Section $4$, we determine the adjacency and Laplacian spectrum of conjugacy and order super commuting graphs of some classes of nonabelian groups.
 
\section{Preliminaries}

Study of graph spectra is one of the most prevalent quests in graph theory. It has applications in various subjects like biology, chemistry, physics, economics, computer science, information and communication technologies, for example see \cite{Tr,VM} and references therein. Different approaches have been considered to study spectral properties of several graphs, one can refer to \cite{Cv} for a comprehensive study of the literature of graph spectra. 

There are several matrices defined on a graph. Among these, the spectrum of adjacency and Laplacian  matrices are studied most frequently. Let $\Gamma$ be a finite simple graph with vertex set $V(\Gamma)=\{v_1,v_2,\ldots,v_n\}$. The \emph{adjacency matrix} of $\Gamma$ is the $n\times n$ matrix $A(\Gamma)=(a_{ij})$  where $a_{ij}=1$ if $v_i \sim v_j$ and $0$ otherwise.  The \emph{Laplacian matrix} $L(\Gamma)$ of $\Gamma$ is defined as $L(\Gamma)=D(\Gamma) - A(\Gamma),$ where $D(\Gamma)={\rm diag}(d_1, d_2, \ldots, d_n)$ with $d_i=deg(v_i)$, $i = 1, 2, \ldots, n$. It is well known that $L(\Gamma)$ is symmetric and positive semidefinite with the smallest eigenvalue $0$. For more on adjacency matrices of graphs, we refer to the book, \cite{CRS,Cv} and for Laplacian matrices of graphs, we refer to the survey papers \cite{RM} and \cite{BM}.
 
Let $B$ be a square matrix. The spectrum of $B$, denoted by $\sigma(B)$, is the multiset of all the eigenvalues of B. If $\lambda_1 \leq \lambda_2 \leq \cdots \leq \lambda_r$ be the distinct eigenvalues of  $B$  with multiplicities $m_1, m_2, \ldots, m_r$, respectively then we shall denote the spectrum of $B$ by   
$\displaystyle \sigma(B)=\begin{pmatrix}
\lambda_1 & \lambda_2 & \cdots& \lambda_r\\
m_1 & m_2 & \cdots & m_r\\
\end{pmatrix}$. It is known that $\displaystyle \sigma(A(K_n))=\begin{pmatrix}
-1 & n-1\\
n-1 & 1 \\
\end{pmatrix}$ and $\displaystyle \sigma(L(K_n))=\begin{pmatrix}
0 & n \\
1 & n-1\\
\end{pmatrix}$.
	
By $\chi(B,x)$ and det$(B)$, we  mean the characteristic polynomial and the determinant of  $B$, respectively. We denote the square matrix of order $n$ having all the entries $1$ by $J_n.$ A principal submatrix of $B$ is a matrix obtained from $B$ by deleting some rows and corresponding columns. The following theorem provides an important relation between the eigenvalues of an $n\times n$ symmetric matrix and its principal submatrices.
	
\begin{theorem}[{\rm\cite{HJ}, Interlacing Theorem}]
Let $B$ be an $n\times n$ real symmetric matrix  and let $m$ be an integer with $1\leq m \leq n.$ Let $B_m$ be an $m\times m$ principal submatrix of $B$. Suppose $\lambda_1 \leq \cdots \leq \lambda_n$ are the eigenvalues of $B$ and $\beta_1 \leq \cdots \leq \beta_m$ are the eigenvalues of $B_m$. Then \[\lambda_k \leq \beta_k \leq \lambda_{k + n -m} ~ \text{for} ~k=1, 2 \dots, m.\]		
In particular if $m = n-1,$ then
		\[\lambda_1 \leq \beta_1 \leq \lambda_{2}\leq \beta_2 \leq \lambda_{3} \leq \cdots \leq \beta_{n - 1} \leq \lambda_{n}.\]
\end{theorem}

To study the characteristic polynomials of different matrices associated with a graph simultaneously, the generalised characteristic polynomial is introduced in \cite{Cv}. 	The \emph{generalized characteristic polynomial} of the graph $\Gamma$ is a bivariate polynomial, defined as  \[\phi_{\Gamma}(x, t) = {\rm det} (xI_n - (A(\Gamma) - tD(\Gamma))).\] The following remark is straightforward.
	
\begin{remark}\label{rem}
Let $\Gamma$ be a graph on  $n$ vertices. Then
		
\begin{itemize}
\item $\chi(A(\Gamma), x) = \phi_{\Gamma}(x, 0);$
			
\item $\chi(L(\Gamma), x) = (-1)^n \phi_{\Gamma}(-x, 1).$			
\end{itemize}
\end{remark} 

The next result tells about the generalized characteristic polynomial of generalized join of a family of regular graphs, which is very useful for the study of the spectrum of super graphs defined on groups. 

\begin{theorem}[{\rm\cite[Theorem 3.1]{Ch}}]\label{Gen-Ch-Poly}
Let $\Im$ be a connected graph with $V(\Im)=\{v_1,v_2,\ldots,v_k\}$ and let $\Gamma=\Im [\Gamma_1, \Gamma_2, \ldots, \Gamma_k]$ be the $\Im$-generalized join graph of  $\Gamma_1, \Gamma_2, \ldots, \Gamma_k$, where $\Gamma_i$ is an $r_i$-regular graph with $n_i$ vertices for $i = 1, \ldots, k$. Then \[\phi_{\Gamma}(x, t) = \chi(N(t), x) \prod\limits_{i = 1}^{k} \frac{\phi_{\Gamma_i}(x + tN_i, t)}{x - r_i + t(r_i + N_i)}, \]
		
		where  $N(t) = \displaystyle \begin{bmatrix}
			r_1 -t(r_1 + N_1) & \sqrt{n_1n_2}\rho_{12} & \cdots & \sqrt{n_1n_k}\rho_{1k}   \\
			\sqrt{n_1n_2}\rho_{12} & r_2 -t(r_2 + N_2) & \cdots & \sqrt{n_2n_k}\rho_{2k}\\
			\; \; \vdots & \; \; \vdots& \ddots & \; \; \vdots \\ 
			\sqrt{n_1n_k}\rho_{1k}  & \sqrt{n_2n_k}\rho_{2k} & \cdots  & r_k -t(r_k + N_k) 
		\end{bmatrix}$ with \\ 
		
		\vspace{0.5cm}
        $N_i = \sum\limits_{v_iv_j \in E(\Im)} n_j$ and  
 		$\rho_{ij} = \rho_{ji} =  \left\{ \begin{array}{ll}
			1 & \mbox{$v_iv_j \in E(\Im)$,}\\
			0& \mbox{otherwise.}\end{array} \right.$ 
\end{theorem}

\section{$\mathcal{R}$-super graph of a graph}

Let $X$ be a nonempty set and let $\mathcal{R}$ be an equivalence relation on $X$. For $x\in X$, we denote the $\mathcal{R}$-equivalence class of $x$  by $[x]_{\mathcal{R}}$. Let $\mathcal{R}_1$ and $\mathcal{R}_2$ be two equivalence relations on $X$. We say $\mathcal{R}_1$ is contained in $\mathcal{R}_2$, denoted by $\mathcal{R}_1\subseteq \mathcal{R}_2$ if  $[x]_{\mathcal{R}_1}\subseteq [x]_{\mathcal{R}_2}$ for any $x\in X$. Let $T_X$ be the set of all equivalence relations on $X$. Then $(T_X, \; \subseteq)$ is a partially ordered set with $\mathcal{R}_l$ and $\mathcal{R}_g$ as the least and the greatest elements respectively, where for any $x\in X$, $[x]_{_{\mathcal{R}_l}}=\{x\}$ and $[x]_{_{\mathcal{R}_g}}=X$. For a given graph $\Gamma$, we will now define some graphs based on the equivalence relation on $V(\Gamma)$.

Let $\Gamma$ be a graph and let $\mathcal{R}$ be an equivalence relation on $V(\Gamma)$. Let $C_1, C_2, \ldots, C_k$ be the distinct $\mathcal{R}$-equivalence classes of $V(\Gamma)$ with $|C_i|=n_i$, for $1\leq i \leq k$. The $\mathcal{R}$-compressed $\Gamma$ graph $\Im_{_{\Gamma^{\mathcal{R}}}}$ is a simple graph with $V(\Im_{_{\Gamma^{\mathcal{R}}}})=\{C_1, C_2, \ldots, C_k\}$ and two distinct vertices $C_i$ and $C_j$ are adjacent if there exist $x \in C_i$ and  $y \in C_j$ such that $x$ is adjacent to $y$ in $\Gamma$. We now define a super graph of $\Gamma$ which is a generalization of super graphs defined on groups. 

The $\mathcal{R}$-super $\Gamma$ graph $\Gamma^{\mathcal{R}}$ is a simple graph with vertex set $V(\Gamma)$ and two distinct vertices are adjacent if either they are in the same $\mathcal{R}$-equivalence class or there are elements in their respective $\mathcal{R}$-equivalence classes that are adjacent in the original graph $\Gamma$. 

Note that if $[v]_{\mathcal{R}}=\{v\}$ for any $v\in V(\Gamma)$ then $\Gamma^{\mathcal{R}}=\Gamma$ and if $[v]_{\mathcal{R}}=V(\Gamma)$ for any $v\in V(\Gamma)$ then $\Gamma^{\mathcal{R}}$ is a complete graph with vertex set $V(\Gamma)$. We have the following results on $\Gamma^{\mathcal{R}}$.
\begin{proposition}\label{poset}
Let $\Gamma$ be a graph and let $\mathcal{R}_1$ and $\mathcal{R}_2$ be two equivalence relations on $V(\Gamma)$. If $\mathcal{R}_1\subseteq \mathcal{R}_2$  then $\Gamma^{\mathcal{R}_1}$ is a spanning subgraph of  $\Gamma^{\mathcal{R}_2}$.
\end{proposition}
\begin{proof} From the construction of super $\Gamma$ graph, it is clear that $V(\Gamma^{\mathcal{R}_1})=V(\Gamma^{\mathcal{R}_2})=V(\Gamma)$. Let $x,y \in V(\Gamma)$ and suppose $x \sim y$ in $\Gamma^{\mathcal{R}_1}$. If $[x]_{_{\mathcal{R}_1}}=[y]_{_{\mathcal{R}_1}}$ then $[x]_{_{\mathcal{R}_2}}=[y]_{_{\mathcal{R}_2}}$ as $\mathcal{R}_1\subseteq \mathcal{R}_2$ and so $x \sim y$ in $\Gamma^{\mathcal{R}_2}$. If $[x]_{_{\mathcal{R}_1}}\neq [y]_{_{\mathcal{R}_1}}$ then there exist $x'\in [x]_{_{\mathcal{R}_1}}$ and $y'\in [y]_{_{\mathcal{R}_1}}$ such that $x'\sim y'$ in $\Gamma$. As $\mathcal{R}_1\subseteq \mathcal{R}_2$ so $x'\in [x]_{_{\mathcal{R}_2}}$ and $y'\in [y]_{_{\mathcal{R}_2}}$ and hence $x \sim y$ in $\Gamma^{\mathcal{R}_2}$. This proves the result. 
\end{proof}
\begin{theorem}\label{obs}
Let $\Gamma$ be a graph and let $\mathcal{R}$ be an equivalence relation on $V(\Gamma)$. Let $C_1, C_2, \ldots, C_k$ be the distinct $\mathcal{R}$-equivalence classes of $V(\Gamma)$ with $|C_i|=n_i$ for $1\leq i \leq k$. Then  $\Gamma^{\mathcal{R}}$ is isomorphic to $\Im_{_{\Gamma^{R}}}[K_{n_1}, K_{n_2},\ldots, K_{n_k}].$
\end{theorem}
\begin{proof}
It is clear from the construction of $\Gamma^{\mathcal{R}}$ that the induced subgraph of $\Gamma^{\mathcal{R}}$ with vertex set $C_i$ is isomorphic to $K_{n_i},$ for $1\leq i \leq k.$ The result follows from the definition of generalized join graph and the construction of $\mathcal{R}$-compressed $\Gamma$ graph.
\end{proof}
Consider the graph $\Gamma$ with $V(\Gamma)=\{v_1,v_2,v_3\}$ and $E(\Gamma)=\{v_2v_3\}$. Let $\mathcal{R}$ be the equivalence relation on $V(\Gamma)$ with equivalence classes $C_1=\{v_1,v_2\}$ and $C_2=\{v_3\}$. Then 
$\Im_{_{\Gamma^{\mathcal{R}}}}$ is isomorphic to $K_2$. In this case, though $\Gamma$ is disconnected but the $\mathcal{R}$-compressed $\Gamma$ graph is connected. The next result tells about how both the graphs $\Gamma$ and $\Im_{_{\Gamma^{\mathcal{R}}}}$ are related through connectedness.
\begin{theorem}\label{comc}
Let $\Gamma$ be a graph and let $\mathcal{R}$ be an equivalence relation on $V(\Gamma)$. Let $C_1, C_2, \ldots, C_k$ be the distinct $\mathcal{R}$-equivalence classes of $V(\Gamma)$ and let $\Gamma_i$ be the induced subgraph of $\Gamma$ with vertex set $C_i$, for $1\leq i \leq k$. If $\Gamma$ is connected then $\Im_{_{\Gamma^{\mathcal{R}}}}$ is connected. Conversely,  if $\Im_{_{\Gamma^{\mathcal{R}}}}$ is connected and all the $\Gamma_i$s are connected then $\Gamma$ is connected.
\end{theorem}
\begin{proof}
Suppose $\Gamma$ is connected. Then as $\Gamma$ is a spanning subgraph of $\Gamma^{\mathcal{R}}$, so $\Gamma^{\mathcal{R}}$ is connected. By Theorem \ref{obs}, $\Gamma^{\mathcal{R}}$ is $\Im_{_{\Gamma^{R}}}$-generalized join of some complete graphs and hence $\Im_{_{\Gamma^{R}}}$ is connected by Lemma \ref{Gj}.

Now suppose $\Im_{_{\Gamma^{\mathcal{R}}}}$ is connected and all the $\Gamma_i$s are connected. Let $x,y\in V(\Gamma)$. If $x,y\in C_i$ for some $i,\; 1\leq i \leq k$ then there is a path between $x$ and $y$ in $\Gamma$ as $\Gamma_i$ is connected. Otherwise let $x\in C_i$ and $y\in C_j$ where $i\neq j$. Since  $\Im_{_{\Gamma^{\mathcal{R}}}}$ is connected, there exist $x'\in C_i$ and $y'\in C_j$ such that there is a path between $x'$ and $y'$ in $\Gamma$. So there is a path between $x$ and $y$ in $\Gamma$ as both $\Gamma_i$ and $\Gamma_j$ are connected.
\end{proof}
 We will now discuss the adjacency and Laplacian characteristic polynomials of the  $\mathcal{R}$-super $\Gamma$ graph $\Gamma^{\mathcal{R}}$. Let $\Gamma$ be a connected graph. Let $\mathcal{R}$ be an equivalence relation on $V(\Gamma)$ with $k$ distinct equivalence classes. If $k\leq 2$ then $\Gamma^{\mathcal{R}}$ is complete. So, for rest of this section, we consider $k\geq 3$.
\begin{theorem}\label{Ch-Su_Gr}
Let $\Gamma$ be a connected graph on $n$ vertices and let $\mathcal{R}$ be an equivalence relation on $V(\Gamma)$. Let $C_1, C_2, \ldots, C_k$ be the distinct $\mathcal{R}$-equivalence classes of $V(\Gamma)$ with $|C_i|=n_i$ for $1\leq i \leq k$. Suppose $A(\Im_{_{\Gamma^{\mathcal{R}}}})=[\rho_{ij}]$ and  $ N_i = \sum\limits_{C_iC_j \in E(\Im_{_{\Gamma^{\mathcal{R}}}})}n_j$, for $1\leq  i \leq  k.$ Then  we have the following:		
\begin{enumerate}[\rm (i)]
\item The characteristic polynomial of  $A(\Gamma^{\mathcal{R}})$ is given by
			
			\[ \chi(A(\Gamma^{\mathcal{R}}), x) = \chi(N(0), x) (x + 1)^{n - k}, \]
			
			where $N(0) = \displaystyle \begin{bmatrix}
				n_1-1  & \sqrt{n_1n_2}\rho_{12} & \cdots & \sqrt{n_1n_k}\rho_{1k}   \\
				\sqrt{n_1n_2}\rho_{12} & n_2-1  & \cdots & \sqrt{n_2n_k}\rho_{2k}\\
				\; \; \vdots & \; \; \vdots& \ddots & \; \; \vdots \\ 
				\sqrt{n_1n_k}\rho_{1k}  & \sqrt{n_2n_k}\rho_{2k} & \cdots  & n_k-1  
			\end{bmatrix}.$\\ 
           
\item The characteristic polynomial of  $L(\Gamma^{\mathcal{R}})$ is given by\[\chi(L(\Gamma^{\mathcal{R}}), x) = \chi(-N(1), x) \prod\limits_{i = 1}^{k}   (x- N_i- n_i)^{n_i - 1},\]
			
\noindent where $N(1) = \displaystyle \begin{bmatrix}
				-	N_1  & \sqrt{n_1n_2}\rho_{12} & \cdots & \sqrt{n_1n_k}\rho_{1k}   \\
				\sqrt{n_1n_2}\rho_{12} & -N_2  & \cdots & \sqrt{n_2n_k}\rho_{2k}\\
				\; \; \vdots & \; \; \vdots& \ddots & \; \; \vdots \\ 
				\sqrt{n_1n_k}\rho_{1k}  & \sqrt{n_2n_k}\rho_{2k} & \cdots  &- N_k  
			\end{bmatrix}.$
  		
\end{enumerate}				
\end{theorem}
\begin{proof}
By Theorem \ref{comc}, $\Im_{_{\Gamma^{\mathcal{R}}}}$ is connected as $\Gamma$ is connected. For $1\leq i \leq k$, the induced subgraph of $\Gamma^{\mathcal{R}}$ with vertex set $C_i$ is isomorphic to $K_{n_i}$ and by Theorem \ref{obs}, $\Gamma^{\mathcal{R}}$ is isomorphic to $\Im_{_{\Gamma^{R}}}[K_{n_1}, K_{n_2},\ldots, K_{n_k}].$

(i) Then in view of Theorem \ref{Gen-Ch-Poly} and Remark \ref{rem}, we have
		\[  \chi(A(\Gamma^{\mathcal{R}}), x) = \chi(N(0), x) \prod_{i = 1}^{k} \frac{\chi(A(K_{n_i}), x)}{x-(n_i - 1)},  \]

        where $N(0) = \displaystyle \begin{bmatrix}
			n_1-1  & \sqrt{n_1n_2}\rho_{12} & \cdots & \sqrt{n_1n_k}\rho_{1k}\\
			\sqrt{n_1n_2}\rho_{12} & n_2-1  & \cdots & \sqrt{n_2n_k}\rho_{2k}\\
			\; \; \vdots & \; \; \vdots& \ddots & \; \; \vdots \\ 
			\sqrt{n_1n_k}\rho_{1k}  & \sqrt{n_2n_k}\rho_{2k} & \cdots  & n_k-1  
		\end{bmatrix}$.\\
  
        Since  the eigenvalues of  $A(K_{n_i})$ are $-1$ and $n_i- 1$ with multiplicities $n_i - 1$ and $1$, respectively, we have 
		\[ \chi(A(\Gamma^{\mathcal{R}}), x) = \chi(N(0), x) (x + 1)^{n_1 + n_2 + \cdots + n_k -k }= \chi(N(0), x) (x + 1)^{n - k}. \]
		
(ii) Again, from Theorem \ref{Gen-Ch-Poly} and Remark \ref{rem}, we have
		
	\[\chi(L(\Gamma^{\mathcal{R}}), x) = (-1)^n \chi(N(1), -x) \prod\limits_{i = 1}^{k} \frac{\phi_{K_{n_i}}(-x + N_i, 1)}{N_i - x}, \]
	
		where  $N(1) = \displaystyle \begin{bmatrix}
		-N_1  & \sqrt{n_1n_2}\rho_{12} & \cdots & \sqrt{n_1n_k}\rho_{1k}   \\
		\sqrt{n_1n_2}\rho_{12} & - N_2  & \cdots & \sqrt{n_2n_k}\rho_{2k}\\
		\; \; \vdots & \; \; \vdots& \ddots & \; \; \vdots \\ 
		\sqrt{n_1n_k}\rho_{1k}  & \sqrt{n_2n_k}\rho_{2k} & \cdots  & -N_k
	\end{bmatrix}.$  
	
	\vspace{0.3cm}
	
	\noindent 
	Also, for $1 \leq i \leq k$, we have

	\[\phi_{K_{n_i}}(-x + N_i, 1) = {\rm det}((-x + N_i)I_{n_i} -A(K_{n_i}) + D(K_{n_i})) \]
	\[ ~~~~~~~~~~~~~ ~~~~~~~ = {\rm det}((-x + N_i + n_i- 1)I_{n_i} -A(K_{n_i}))\]
	\[~~~~~~~~~~~~~~~~~~~~~~~~~~~~~~ = (-1)^{n_i} {\rm det}(xI_{n_i} - ((N_i + n_i -1)I_{n_i}- A(K_{n_i})))\]
	\[~~~~~~~~~~~~~~~~~~~~ = (-1)^{n_i} {\rm det}(xI_{n_i} - ((N_i + n_i)I_{n_i}- J_{n_i})).\]
	
	\noindent The eigenvalues of the matrix $J_{n_i}$ are $0$ and $n_i$ with multiplicities $n_i-1$ and $1$ respectively. So, the eigenvalues of the matrix $(N_i + n_i)I_{n_i}- J_{n_i}$ are $N_i + n_i$ and $N_i$ with multiplicities $n_i-1$ and $1$, respectively. Thus, we have 
	
	\[\phi_{K_{n_i}}(-x + N_i, 1) = (-1)^{n_i} (x - N_i) (x- N_i- n_i)^{n_i - 1}. \]

  \noindent Hence, \\
$\begin{aligned}
	\chi(L(\Gamma^{\mathcal{R}}), x) &= (-1)^{n} \chi(N(1), -x) \prod\limits_{i = 1}^{k} \frac{\phi_{K_{n_i}}(-x + N_i, 1)}{N_i - x}\\
	                                 &=(-1)^{n} (-1)^k \chi(-N(1), x) \prod\limits_{i = 1}^{k} \frac{(-1)^{n_i} (x - N_i) (x- N_i- n_i)^{n_i - 1}}{N_i -x}\\
                                     &= \chi(-N(1), x) \prod\limits_{i = 1}^{k}   (x- N_i- n_i)^{n_i - 1}.
\end{aligned}$\\

This completes the proof.
	
\end{proof}

For the star graph $K_{1,k-1}$ with vertex set $\{v_1,v_2,\ldots,v_k\}$ and  $k\geq 2$, we consider $v_1$ as the central vertex, $i.e.$ deg$(v_1)=k-1$. Then we have the following corollaries which are very useful. 

\begin{corollary}\label{Spec}
Let $\Gamma$ be the generalized join graph $K_{1,k-1}[K_{n_1},K_{n_2},\ldots, K_{n_k}]$ with $n=n_1+n_2+\cdots+n_k$. Then the characteristic polynomial of $A(\Gamma)$ is given by
		\[\chi(A(\Gamma), x) = (x + 1)^{n -k} \{\prod_{i = 1}^{k}(x- n_i + 1) -  n_1n_2 \prod_{i = 3}^{k}(x- n_i + 1) - \cdots \] \[~~~~~~~~~~ -n_1n_l\prod_{i = 2,~ i \neq l}^{k}(x- n_i + 1) - \cdots - n_1n_k\prod_{i = 2}^{k-1}(x- n_i + 1)\}.\]
\end{corollary}
	
\begin{proof} Let $A(K_{1,k-1})=[\rho_{ij}]$. Then for $2\leq j\leq k$, $\rho_{1j}= \rho_{j1}=  1$  and for  $2 \leq i, j \leq k,\; \; i\neq j,$       $\rho_{ij} =  0$.   By  Theorem \ref{Ch-Su_Gr}(i), we have
		\[ \chi(A(\Gamma), x) = \chi(N(0), x) (x + 1)^{n - k}, \]
		
		where  
		\[\chi(N(0), x) = \displaystyle \begin{vmatrix}
			x -n_1 +1  & - \sqrt{n_1n_2} & - \sqrt{n_1n_3} & \cdots & - \sqrt{n_1n_k}   \\
			- \sqrt{n_1n_2} & x- n_2 + 1 & 0  & \cdots & 0 \\
			- \sqrt{n_1n_3} & 0 & x- n_3 + 1  & \cdots & 0 \\
			\; \; \vdots & \; \; \vdots & \; \; \vdots& \ddots & \; \; \vdots \\ 
			- \sqrt{n_1n_k}  & 0 & 0  & \cdots  & x - n_{k} + 1  
		\end{vmatrix}.\]
		
		Expanding the above determinant by first row, we get \\
		\[\chi(N(0), x) = (x -n_1 +1)(x- n_2 + 1) (x- n_3 + 1) \cdots ( x - n_k + 1 ) + \] 
		
		\[ \sqrt{n_1n_2} \displaystyle \begin{vmatrix}
			-\sqrt{n_1n_2}  & 0  & \cdots & 0 \\
			-\sqrt{n_1n_3} & x- n_3 + 1  & \cdots & 0 \\
			\; \; \vdots &  \; \; \vdots& \ddots & \; \; \vdots \\ 
			-\sqrt{n_1n_k}  & 0  & \cdots  & x - n_k + 1  
		\end{vmatrix} +  \]
		
		\[\cdots + (-1)^{l}\sqrt{n_1n_l} \displaystyle \begin{vmatrix}
			-\sqrt{n_1n_2} & x- n_2 + 1 & 0  & \cdots & 0 \\
			-\sqrt{n_1n_3} & 0 & x- n_3 + 1  & \cdots & 0 \\
			\; \; \vdots & \; \; \vdots & \; \; \vdots& \ddots & \; \; \vdots \\ 
			-\sqrt{n_1n_l}  & 0 & 0  & \cdots  & 0\\
			\; \; \vdots & \; \; \vdots & \; \; \vdots& \ddots & \; \; \vdots \\ 
			-\sqrt{n_1n_k}  & 0 & 0  & \cdots  & x - n_k + 1  
		\end{vmatrix} +  \cdots + \]	
\vspace{0.5cm}
		
		\[ (-1)^{k}\sqrt{n_1n_k} \displaystyle \begin{vmatrix}
			-\sqrt{n_1n_2} & x- n_2 + 1 & 0  &  \cdots & 0 \\
			-\sqrt{n_1n_3} & 0 & x- n_3 + 1  &\cdots &  0 \\
			\; \; \vdots & \; \; \vdots & \; \; \vdots&  \ddots &  \vdots \\ 
			-\sqrt{n_1n_l}  & 0 & 0  & \cdots  & 0 \\
			\; \; \vdots & \; \; \vdots & \; \; \vdots& \ddots & \vdots \\
			-\sqrt{n_1n_{k-1}}  & 0 & 0  & \cdots  & x - n_{k - 1} + 1   \\
			-\sqrt{n_1n_k}  & 0 & 0  & \cdots & 0    
		\end{vmatrix}.\]
		
\vspace{0.3cm}
		
\noindent Solving these determinants, we get\\
\[\chi(N(0), x) = \prod_{i = 1}^{k}(x- n_i + 1) -  n_1n_2 \prod_{i = 3}^{k}(x- n_i + 1) - \cdots \] \[~~~~~~~~~~ -n_1n_l\prod_{i = 2,~ i \neq l}^{k}(x- n_i + 1) - \cdots - n_1n_k\prod_{i = 2}^{k-1}(x- n_i + 1).\] Hence the result follows.		
\end{proof}
\begin{corollary}
Let $\Gamma$ be the generalized join graph $K_{1,k-1}[K_{l},K_{m},\ldots, K_{m}]$. Then $$ \sigma(A(\Gamma)) = \displaystyle \begin{pmatrix}
	\frac{m+l-2-\sqrt{m^2+l^2+(4k-6)ml}}{2} & -1 & m-1 & \frac{m+l-2+\sqrt{m^2+l^2+(4k-6)ml}}{2}\\
	1 & m(k-1)+l-k & k-2 & 1\\
			\end{pmatrix}.$$
\end{corollary}
\begin{proof}
By Corollary \ref{Spec}, we have \\ 
$\begin{aligned}
\chi(A(\Gamma), x) &= (x + 1)^{m(k-1)+l-k} \{(x-l+1)(x-m+1)^{k-1} - (k-1)ml(x-m+1)^{k-2}\}\\
                   &=(x + 1)^{m(k-1)+l-k}(x-m+1)^{k-2}\{(x-l+1)(x-m+1)^2 - (k-1)ml\}\\
                   &=(x + 1)^{m(k-1)+l-k}(x-m+1)^{k-2}(x-\alpha)(x-\beta),
\end{aligned}$ \\
where $\alpha=\frac{m+l-2+\sqrt{m^2+l^2+(4k-6)ml}}{2}$ and $\beta=\frac{m+l-2-\sqrt{m^2+l^2+(4k-6)ml}}{2}.$ Hence the result follows.
\end{proof}
\begin{corollary}\label{Spec_Lap}
Let $\Gamma$ be the generalized join graph $K_{1,k-1}[K_{n_1},K_{n_2},\ldots, K_{n_k}]$ with $n=n_1+n_2+\cdots+n_k$. Then the characteristic polynomial of $L(\Gamma)$ is given by		
		$$\chi(L(\Gamma), x) = x(x - n)(x-n_1)^{k-2}\prod\limits_{i = 1}^{k}(x- N_i- n_i)^{n_i - 1},$$
where $ N_i = \sum\limits_{v_iv_j \in E(K_{1,k-1})}n_j$, for $1\leq  i \leq  k.$
\end{corollary}
	
\begin{proof} We have $N_1 = n - n_1$ and for  $2 \leq i \leq k$, $N_i = n_1$. Then by  Theorem \ref{Ch-Su_Gr},
 $$\chi(L(\Gamma^{\mathcal{R}}), x) = \chi(-N(1), x) \prod\limits_{i = 1}^{k}   (x- N_i- n_i)^{n_i - 1}$$
        where \[N(1) = \displaystyle \begin{bmatrix}
		-(|V(\Gamma)| - n_1)  & \sqrt{n_1n_2} & \cdots & \sqrt{n_1n_k}   \\
		\sqrt{n_1n_2} & - n_1  & \cdots & 0\\
		\; \; \vdots & \; \; \vdots& \ddots & \; \; \vdots \\ 
		\sqrt{n_1n_k}  & 0 & \cdots  & -n_1
	\end{bmatrix}.\]

\noindent So, $\chi(-N(1), x) = \displaystyle \begin{vmatrix}
	x -(n - n_1)  &  \sqrt{n_1n_2} &  \sqrt{n_1n_3} & \cdots &  \sqrt{n_1n_k}   \\
	\sqrt{n_1n_2} & x- n_1 & 0  & \cdots & 0 \\
	 \sqrt{n_1n_3} & 0 & x- n_1  & \cdots & 0 \\
	\; \; \vdots & \; \; \vdots & \; \; \vdots& \ddots & \; \; \vdots \\ 
	 \sqrt{n_1n_k}  & 0 & 0  & \cdots  & x - n_1  
\end{vmatrix} $ 

and by expanding the determinant, we get \\

    $\begin{aligned}
     \chi(-N(1), x) &= (x -n+ n_1)(x-n_1)^{k-1} - n_1 n_2 (x-n_1)^{k-2} - \cdots - -n_1n_k(x-n_1)^{k-2}\\
     &= (x-n_1)^{k-2}[(x -(n- n_1))(x-n_1) - n_1 (n_2 + \cdots + n_k)]\\
     &= (x-n_1)^{k-2}[(x -(n- n_1))(x-n_1) - n_1 (n- n_1)]\\
     &= (x-n_1)^{k-2}(x -n)x.
     \end{aligned}$\\

Thus,
$\chi(L(\Gamma), x) = x(x -n)(x-n_1)^{k-2}\prod\limits_{i = 1}^{k}(x- N_i- n_i)^{n_i - 1}.$ \\
\end{proof}
The next section is devoted to construction of some classes of graphs which are isomorphic to star generalize join of some complete graphs. 

\section{Super commuting graph}
	
Let $G$ be a finite group with the identity element $e$. For $a\in G$, by $\langle a \rangle$ we mean the cyclic subgroup of $G$ generated  by $a$. The \textit{commuting graph} $\Delta(G)$ of $G$ is a simple graph with vertex set $G$ and two distinct vertices are adjacent whenever they commute. The commuting graph of $G$ was introduced by Brauer and Fowler in \cite{BF} where they have considered vertex set as $G - e.$ Later some authors have considered $\Delta(G)$ with vertex set $G-Z(G)$, where $Z(G)$ is the center of $G$. There are many graphs defined on $G$ with vertex set as the whole group $G$. To compare different graphs defined on $G$ and study their hierarchy, recently people have started considering commuting graph $\Delta(G)$ with vertex set as whole $G$ (see \cite{Ca}). So we have considered commuting graph $\Delta(G)$ with vertex set as whole $G$. For more on commuting graphs on different algebraic structures, we refer to \cite{ABJ,AK1,DB, GP, JK, Sh, Wo}.

For the study of super commuting graph on $G$, we consider the following two equivalence relation on $G$:	
\begin{enumerate}
\item[(i)] \emph{conjugacy relation}, $(x, y) \in \mathcal{R}_c$ if and only if $x = gyg^{-1}$ for some $g \in G$;

\item[(ii)] \emph{order relation}, $(x, y) \in \mathcal{R}_o$ if and only if $o(x) = o(y)$, where $o(g)$ denotes the order of  $g\in G$.
\end{enumerate}
We denote the $\mathcal{R}_c$-super commuting graph on $G$ by $\Delta^c(G)$ and the $\mathcal{R}_o$-super commuting graph by $\Delta^o(G).$ For $x,y\in G$, if $x$ is conjugate to $y$ in $G$ then $o(x)=o(y)$. So, $\mathcal{R}_c \subseteq \mathcal{R}_o$ and hence by Theorem \ref{poset}, $\Delta^c(G)$ is a spanning subgraph of $\Delta^o(G).$ The study about the graphs $\Delta^c(G)$ and  $\Delta^o(G)$ is started by Arunkumar \emph{et al.} in \cite{AC} and then continued in \cite{DA}. For an abelian group $G$ of order $n$, $\Delta(G)\cong K_n$ and hence $\Delta^c(G)=\Delta^o(G)\cong K_n.$ So we consider nonabelian groups only. For our study, we examine the following three classes of groups:

\begin{enumerate}
\item[(i)] \emph{Dihedral group:} $D_{2n} = \langle a, b \; : \; a^{n} = b^2 = e, \; ab = ba^{-1} \rangle, n\geq 3;$

\item[(ii)] \emph{Generalized quaternion group:} $Q_{4n} = \langle a, b \; : \; a^{2n} = e, \; b^2 = a^n, \; ab = ba^{-1} \rangle, n\geq 2;$

\item[(ii)] \emph{Nonabelian group of order $pq$($p$ and $q$ are distinct primes with $q|p-1$):} \[\mathbb Z_p \rtimes \mathbb Z_q = \langle a, b : b^p = a^q = e, aba^{-1} = b^m \; with \; m^q \equiv 1 ( {\rm mod}~ p)\; and \;  m \not\equiv 1 ( {\rm mod}~ p)\rangle.\]  
\end{enumerate}
The generalized quaternion group is also known as dicyclic group and the group $\mathbb Z_p \rtimes \mathbb Z_q$ is  the nontrivial semidirect product of the cyclic groups $\mathbb Z_p$ and $\mathbb Z_q.$ Element wise we can express the above groups as following:
\begin{enumerate}
\item[(i)] $D_{2n}=\{e\}\cup\{a^i:1\leq i \leq n-1\}\cup \{ba^i: 0\leq i \leq n-1\};$
\item[(ii)]$Q_{4n}=\{e\}\cup\{a^i:1\leq i \leq 2n-1\}\cup \{ba^i: 0\leq i \leq 2n-1\};$
\item[(iii)]$\mathbb Z_p \rtimes \mathbb Z_q=\{e\}\cup\{b^i:1\leq i \leq p-1\}\cup \{b^ia^j: 0\leq i \leq p-1,1\leq j \leq q-1 \}.$
\end{enumerate}
 It is important to know the order of the elements of the above groups to study the order super commuting graph of these groups. In $D_{2n},$ for $0\leq i \leq n-1,$ we have $o(ba^i)=2$ and $o(a^i)=\dfrac{n}{gcd(i,n)}.$ In the group $Q_{4n},$ $o(ba^i)=4$ and $o(a^i)=\dfrac{2n}{gcd(i,2n)}$ for $0\leq i\leq 2n-1$. Since there is a unique Sylow $p$- subgroup of $\mathbb Z_p \rtimes \mathbb Z_q$, we have $o(b^i)=p$ for $1\leq i \leq p-1$ and $o(b^ia^j)=q$ for $0\leq i \leq p-1,1\leq j \leq q-1.$

\subsection{Order super commuting graph}

In this subsection, we will first discuss the adjacency and Laplacian spectrum of $\Delta^o(G),$ for $G=D_{2n}, Q_{4n},\mathbb Z_p \rtimes \mathbb Z_q.$ 

First consider $n=2m$ for some positive integer $m$. In $D_{2n}$, let $C_1=\{ba^i: 0\leq i \leq n-1\}\cup \{a^m\}$ and $C_2=\langle a \rangle -\{e,a^m\}$. The induced subgraphs of $\Delta^o(D_{2n})$ with vertex set $C_1$ and $C_2$ are complete as $C_1$ is the order $2$ equivalence class and $C_2$ is a subset of the cyclic group $\langle a \rangle$. As $a^m$ is adjacent to all the elements of $C_2$ in $\Delta^o(D_{2n})$ so every element of $C_1$ is adjacent with each elements of $C_2$. Since $e$ is adjacent to all other vertices of $\Delta^o(D_{2n})$ so $\Delta^o(D_{2n})\cong K_{2n}$. Similarly, one can check that $\Delta^o(Q_{4n})\cong K_{4n}$ when $n$ is even.
\begin{theorem}\label{thm:adj}
Let $n\geq 3$ be odd and let $p$ and $q$ be distinct primes with $q|p-1$. Then 
\begin{enumerate}[\rm (i)]
\item $\chi(A(\Delta^o(D_{2n}),x)=(x+1)^{2n - 3}[x^3 - x^2 (2n -3) + x (n^2 - 5n + 3) + (2n^2 - 4n + 1)].$ Furthermore, $ \sigma(A(\Delta^o(D_{2n}))) =  \displaystyle \begin{pmatrix}
	\alpha & -1 & \beta & \gamma\\
	1 & 2n-3 & 1 & 1\\
			\end{pmatrix}$,
where $-2 < \alpha < -1,~ n -2 < \beta < n - 1$ and $n  < \gamma < n+1.$
			
\item $\chi(A(\Delta^o(Q_{4n})),x) =(x+1)^{4n - 3} [x^3 - x^2 (4n -3) + x (4n^2 - 12n + 3) + (12n^2 - 16n + 1)].$ Furthermore, $ \sigma(A(\Delta^o(Q_{4n}))) = \displaystyle \begin{pmatrix}
	\alpha & -1 & \beta & \gamma\\
	1 & 4n-3 & 1 & 1\\
			\end{pmatrix}$,
where $-3 < \alpha < -2,~ 2n -3 < \beta < 2n - 2$ and  $2n+1  < \gamma < 2n+2$ for $3\leq n \leq 13$, otherwise $2n+2  < \gamma < 2n+3.$
			
\item $\chi(A(\Delta^o(\mathbb Z_p \rtimes \mathbb Z_q)),x)=(x+1)^{pq- 3} [x^3 - x^2(pq - 3) + x(p^2q - 3pq - p^2 + p + 3) + (2p^2q - 3pq - 2p^2 + 2p + 1)].$ Furthermore, $ \sigma(A(\Delta^o(\mathbb Z_p \rtimes \mathbb Z_q))) =  \displaystyle \begin{pmatrix}
	\alpha & -1 & \beta & \gamma\\
	1 & 2n-3 & 1 & 1\\
			\end{pmatrix}$,
where $-2 < \alpha < -1,~ p -2 < \beta < p - 1$ and $pq-p  < \gamma < pq-p+1$.
\end{enumerate}
\end{theorem} 
\begin{proof}(i) For the group $D_{2n},$ let  $C_1=\{ba^i: 0\leq i \leq n-1\}$ and $C_2=\langle a \rangle -\{e\}$. As $n$ is odd, so $C_1$ is the order $2$ equivalence class and hence the induced subgraph of $\Delta^o(D_{2n})$ with vertex set $C_1$ is isomorphic to $K_n$. Also the induced subgraph of $\Delta^o(D_{2n})$ with vertex set $C_2$ is isomorphic to $K_{n-1}$ as $C_2$ is a subset of the cyclic group $\langle a \rangle$. There is no edge between $C_1$ and $C_2$ in $\Delta^o(D_{2n})$ as no element of $C_1$ commutes with any element of $C_2$. Since $e$ is adjacent with all other element of $D_{2n}$ so $\Delta^o(D_{2n})$ is isomorphic to the generalized join graph $K_{1,2}[K_1, K_{n-1}, K_n]$. Then by Corollary \ref{Spec}, we get
			\[\chi(A(\Delta^o(D_{2n})), x) ~= (x+1)^{2n - 3} [x (x -n + 2)(x - n+1) -(n-1)(x- n + 1) - n(x-n + 2)] \]
			\[ ~~~~~~~~~= (x+1)^{2n - 3} [x^3 - x^2 (2n -3) + x (n^2 - 5n + 3) + (2n^2 - 4n + 1)].\]
			
By deleting the first row and first column of $A(\Delta^o(D_{2n}))$, we obtain a block diagonal matrix $B$ whose two blocks are the adjacency matrices of $K_{n-1}$ and $K_n$, respectively. Note that the eigenvalues of the matrix $B$ are $\beta_1 = \beta_2 = \cdots = \beta_{2n -3} = -1, \beta_{2n-2} = n -2$ and $\beta_{2n-1} = n -1$. Let $\lambda_1\leq \lambda_2\leq \cdots \leq \lambda_{2n-1}\leq \lambda_{2n}$ be the eigenvalues of $A(\Delta^o(D_{2n}))$. Then by Interlacing theorem, we get \[\lambda_1 \leq \lambda_2 = \lambda_3 = \cdots = \lambda_{2n - 3} = -1 \leq \lambda_{2n - 2} \leq n - 2 \leq \lambda_{2n - 1} \leq n - 1 \leq \lambda_{2n}.\]
			
Let $f(x)=x^3 - x^2 (2n -3) + x (n^2 - 5n + 3) + (2n^2 - 4n + 1)$. Then $f(-1)=n^2-n>0$ and $f(-2)=-(2n+1)<0$ as $n\geq 3$. Since the characteristic polynomial of  $A(\Delta^o(D_{2n}))=(x+1)^{2n-3}f(x)$, then by Intermediate value theorem we get  $\lambda_{2n - 2} = -1$ and $-2 <\lambda_1 < -1$. It can be checked that $f(n-2)=n-1>0,f(n-1)=-n<0,f(n)=-n+1<0$ and $f(n+1)=n+8>0$.  Thus the result follows from Intermediate value theorem.\\

\noindent(ii) For the group $Q_{4n},$ let $C_1=\{e,a^n\}, C_2=\langle a \rangle -\{e, a^n\}$ and $C_3=\{ba^i: 0\leq i \leq 2n-1\}.$ As $n$ is odd, so $C_3$ is the order $4$ equivalence class and hence the induced subgraph of $\Delta^o(Q_{4n})$ with vertex set $C_3$ is isomorphic to $K_{2n}.$ The induced subgraph of $\Delta^o(Q_{4n})$ with vertex set $C_2$ is isomorphic to $K_{2n-2}$ as $C_2$ is a subset of the cyclic group $\langle a \rangle$. There is no edge between $C_2$ and $C_3$ in $\Delta^o(D_{2n})$ as no element of $C_2$ commutes with any element of $C_3$. Also $a^n$ is the only element of order $2$ and it commutes with all other element of $Q_{4n}.$ So $\Delta^o(Q_{4n})$ is isomorphic to the generalized join graph $K_{1,2}[K_2, K_{2n-2}, K_{2n}]$ as $e$ is adjacent with all other elements. Then by Corollary \ref{Spec}, we get
			\[\chi(A(\Delta^o(Q_{4n})),x) = (x+1)^{4n - 3} [(x-1)(x -2n + 3)(x - 2n+1) -4(n -1)(x- 2n + 1) - 4n(x-2n + 3)] \]		
			\[ ~~~~~~~~~~~~~~~~ = (x+1)^{4n - 3} [x^3 - x^2 (4n -3) + x (4n^2 - 12n + 3) + (12n^2 - 16n + 1)].\]
By deleting the first two rows and first two columns of $A(\Delta^o(Q_{4n}))$, we obtain a block diagonal matrix $B$ whose two blocks are the adjacency matrices of $K_{2n-2}$ and $K_{2n}$, respectively.  The eigenvalues of the matrix $B$ are $\beta_1 = \beta_2 = \cdots = \beta_{4n -4} = -1, \beta_{4n-3} = 2n -3$ and $\beta_{4n-2} = 2n -1$. Let $\lambda_1\leq \lambda_2\leq \cdots \leq \lambda_{4n-1}\leq \lambda_{4n}$ be the eigenvalues of $A(\Delta^o(D_{2n}))$. Then by Interlacing theorem, we get $2n-3=	\beta_{4n-3}\leq \lambda_{4n-1}$ and $2n-1=	\beta_{4n-2}\leq \lambda_{4n}.$		
			
Let $f(x)=x^3 - x^2 (4n -3) + x (4n^2 - 12n + 3) + (12n^2 - 16n + 1).$ Then $f(-2)=4n^2-8n-1>0$ and $f(-3)=-8(2n+1)<0$ as $n\geq 3$.	 Since the characteristic polynomial of  $A(\Delta^o(Q_{4n}))=(x+1)^{4n-3}f(x)$, then by Intermediate value theorem we get  $-3 <\lambda_1 < -2$ and $\lambda_2= \lambda_3= \cdots = \lambda_{4n-2}=-1.$ It can be checked that $f(2n-3)=8(n-1)>0,f(2n-2)=-2n-1<0,f(2n+1)=-8n+8<0, f(2n+2)=-2n+27$ and $f(2n+3)=8(n+8)>0.$ As $f(2n+2)>0$ for $3\leq n\leq 13$ and $f(2n+2)<0$ for $n\geq 15$ so  the result follows from Intermediate value theorem.\\

\noindent(iii) For the group  $\mathbb Z_p \rtimes \mathbb Z_q$, let $C_2=\langle b \rangle -\{e\}$ and $C_3=\{b^ia^j: 0\leq i \leq p-1,1\leq j \leq q-1 \}.$ Then $C_2$ is the order $p$ equivalence class and $C_3$ is the order $q$ equivalence class. Hence the induced subgraph of $\Delta^o(\mathbb Z_p \rtimes \mathbb Z_q)$ with vertex set $C_2$ and $C_3$ are isomorphic to $K_{p-1}$ and $K_{pq-p},$ respectively. There is no edge between $C_2$ and $C_3$ in $\Delta^o(\mathbb Z_p \rtimes \mathbb Z_q)$ as no element of $C_2$ commutes with any element of $C_3$. Since $e$ is adjacent with all other element of  $\mathbb Z_p \rtimes \mathbb Z_q$ so $\Delta^o(\mathbb Z_p \rtimes \mathbb Z_q)$ is isomorphic to the generalized join graph $K_{1,2}[K_1, K_{p-1}, K_{pq-p}]$. Then by Corollary \ref{Spec}, we get		
			\[\chi(A(\Delta^o(G)),x) = (x+1)^{pq - 3} [x (x -p + 2)(x - pq + p + 1) -(p - 1)(x- pq + p + 1) - (pq - p)(x- p + 2)] \]
			\[~ = (x+1)^{pq- 3} [x^3 - x^2(pq - 3) + x(p^2q - 3pq - p^2 + p + 3) + (2p^2q - 3pq - 2p^2 + 2p + 1)]. \]
			
The rest part of the proof is similar to the proof of $(i).$
\end{proof} 
\begin{theorem}\label{thm:lap}
Let $n\geq 3$ be odd and let $p$ and $q$ be distinct primes with $q|p-1$. Then 
\begin{enumerate}[\rm (i)]
\item  $ \sigma(L(\Delta^o(D_{2n}))) =  \displaystyle \begin{pmatrix}
	0 & 1 & n & n+1 & 2n\\
	1 & 1 & n-2 & n-1 & 1\\
			\end{pmatrix}$;

\item  $ \sigma(L(\Delta^o(Q_{4n}))) = \displaystyle \begin{pmatrix}
	0 & 2 & 2n & 2n+2 & 4n\\
	1 & 1 & 2n-3 & 2n-1 & 2\\
			\end{pmatrix}$;

\item  $ \sigma(L(\Delta^o(\mathbb Z_p \rtimes \mathbb Z_q))) =  \displaystyle \begin{pmatrix}
	0 & 1 & p & pq-p+1 & pq\\
	1 & 1 & p-2 & pq-p-1 & 1\\
			\end{pmatrix}$.
\end{enumerate}
\end{theorem} 
\begin{proof}
By the proof Theorem \ref{thm:adj}, we have $\Delta^o(D_{2n}) \cong K_{1,2}[K_1, K_{n-1}, K_n]$, $\Delta^o(Q_{4n})\cong K_{1,2}[K_2, K_{2n-2}, K_{2n}]$ and $\Delta^o(\mathbb Z_p \rtimes \mathbb Z_q)\cong K_{1,2}[K_1, K_{p-1}, K_{pq-p}]$. Then the result follows from Corollary \ref{Spec_Lap}.
\end{proof}

\subsection{Conjugacy super commuting graph}

In this subsection, our aim is to study the conjugacy super commuting graph $\Delta^c(G)$ , where $G\in \{D_{2n},Q_{4n}, \mathbb Z_p \rtimes \mathbb Z_q \}.$ 

\begin{theorem}\label{Ad_D_2n}
Let $n\geq 3$ be a positive integer. Then $\Delta^c(D_{2n})\cong K_{1,3}[K_2, K_{\frac{n}{2}}, K_{\frac{n}{2}}, K_{n-2}]$ if $n$ is even and $\Delta^c(D_{2n})\cong K_{1,2}[K_1, K_{n-1}, K_n]$ if $n$ is odd.
\end{theorem}
\begin{proof} First let $n=2k$ with $k\geq 2.$ Then it is known that the conjugacy classes of Dihedral group $D_{2n}$ are
\[\{e\}, \{a^k\}, \{a^{\pm 1}\}, \{a^{\pm 2}\}, \ldots, \{a^{\pm (k-1)}\}, \{a^{2i}b : 1 \leq i \leq k\}, ~ {\rm and}~ \{a^{2i - 1}b : 1 \leq i \leq k\}.\]

Take $C_1=\{a^{2i}b : 1 \leq i \leq k\}, C_2=\{a^{2i - 1}b : 1 \leq i \leq k\} ~ {\rm and}~ C_3=\langle a \rangle - \{e, a^k\}.$ Then for $j\in \{1,2,3\}$, the induced subgraph of $\Delta^c(D_{2n})$ with vertex set $C_j$ is complete. For $j_1,j_2\in \{1,2,3\}$ with $j_1\neq j_2$, note that no element of $C_{j_1}$ commutes with any element of $C_{j_2}$ and hence in $\Delta^c(D_{2n})$, no vertex of $C_{j_1}$ is adjacent with any vertex of $C_{j_2}$. As $e$ and $a^k$ commute with every element of $D_{2n}$, so $\Delta^c(D_{2n})\cong K_{1,3}[K_2, K_{\frac{n}{2}}, K_{\frac{n}{2}}, K_{n-2}]$.

Let $n = 2k + 1$ with $k\geq 1$. Then the conjugacy classes of $D_{2n}$ are
\[\{e\}, \{a^{\pm 1}\}, \{a^{\pm 2}\}, \ldots, \{a^{\pm k}\}, ~ {\rm and}~ \{a^{i}b : 1 \leq i \leq n\}.\]
Take $C'_1=\langle a \rangle - \{ e \}$ and $C'_2=\{a^{i}b : 1 \leq i \leq n\}$. Then the induced subgraphs of $\Delta^c(D_{2n})$ with vertex set $C'_1$ and $C'_2$ are complete graphs. As no element of $C'_1$ commutes with any element of $C'_2$, we have $\Delta^c(D_{2n})\cong K_{1,2}[K_1, K_{n-1}, K_n]$.
\end{proof}
If $n$ is odd then by Theorem \ref{Ad_D_2n}, $\Delta^c(D_{2n})=\Delta^o(D_{2n})$ and in this case the spectrum of $\Delta^c(D_{2n})$ are already discussed. For $n=2m \; (m\geq 2)$, we will relate $\Delta^c(D_{2n})$ and $\Delta^c(Q_{4m})$ in the next result.

\begin{theorem}\label{Ad_Q4n}
Let $n\geq 2$ be a positive integer. Then $\Delta^c(Q_{4n})\cong K_{1,3}[K_2, K_n, K_n, K_{2n-2}].$ 
\end{theorem}
\begin{proof}
It is easy to verify that the conjugacy classes of $Q_{4n}$ are 
\[\{e\}, \{a^n\}, \{a^{\pm 1}\}, \{a^{\pm 2}\}, \ldots, \{a^{\pm (n-1)}\}, \{a^{2i}b : 1 \leq i \leq n\}, ~ {\rm and}~ \{a^{2i - 1}b : 1 \leq i \leq n\}.\]
The result follows by using a similar argument as given in Theorem \ref{Ad_D_2n}.
\end{proof}
By Theorems \ref{Ad_D_2n} and \ref{Ad_Q4n}, for $n=2m \;(m\geq 2)$, $\Delta^c(D_{2n}) \cong \Delta^c(Q_{4m}).$	By Corollary \ref{Spec_Lap}, we have $\displaystyle \sigma(L(\Delta^c(D_{4m})))=\sigma(L(\Delta^c(Q_{4m})))=\begin{pmatrix}
0 & 2 & m+2 & 2m & 4m\\
1 & 1 & 2m-2 & 2m-3 & 2\\
\end{pmatrix}$ and using Corollary \ref{Spec}, we get $\chi(A(\Delta^c(D_{4m})),x)=\chi(A(\Delta^c(Q_{4m})),x)=(x+1)^{4m-4}(x-m+1)[x^3-(3m-3)x^2+(2m^2-10m+3)x+(10m^2-15m+1)].$
\begin{theorem}\label{semi_Ad}
Let $p$ and $q$ be distinct primes with $q \mid p-1$. Then $\Delta^c(\mathbb Z_p \rtimes \mathbb Z_q)\cong K_{1,2}[K_1, K_{p-1}, K_{pq-p}].$
\end{theorem}
\begin{proof} We have $\mathbb Z_p \rtimes \mathbb Z_q = \langle a, b : b^p = a^q = e, aba^{-1} = b^m \rangle$ with $m^q \equiv 1 ( {\rm mod}~ p)$ and   $m \not\equiv 1 ( {\rm mod}~ p).$ As $ab = b^m a$, we get
\[ab^i = b^{mi} a, ~ a^i b = b^{m^i} a^i, ~ {\rm and} ~ a^{i}b^j = b^{mj}a^i. \]

Let $x\in \mathbb Z_p \rtimes \mathbb Z_q.$ Then $x = b^i a^j$ for some $ 0 \leq i \leq p-1$ and $0 \leq j \leq q-1$. Let $CL(x)$ be the conjugacy class of $x$ and let $y\in CL(x).$ Then by using the above equality, we have 
$$y=b^ka^l(x)a^{-l}b^{-k} = b^k a^l b^i a^{j-l}b^{-k} = b^k b^{m^li} a^j b^{-k} = b^{m^li- m^jk + k} a^j.$$

For $1\leq j \leq q-1$, the normalizer of $a^j$, $N(a^j)= \langle a \rangle$. Then $|CL(a_j)|=\dfrac{|\mathbb Z_p \rtimes \mathbb Z_q|}{|N(a^j)|}=p$. So, $1\leq j \leq q-1$, $CL(a_j) = \{b^l a^j : 0 \leq i \leq p\}.$ Let $C_j=CL(a_j)$ and let $j_1,j_2\in \{1,2,\dots, q-1\}$ with $j_1\neq j_2$. Then every element of $C_{j_1}$ is adjacent with every element of $C_{j_2}$ in $\Delta^c(\mathbb Z_p \rtimes \mathbb Z_q)$ as $a_{j_1}a_{j_2}=a_{j_2}a_{j_1}.$ Let $H=C_1\cup C_2 \cup \cdots \cup C_{q-1}$. Then the induced subgraph of $\Delta^c(\mathbb Z_p \rtimes \mathbb Z_q)$ with vertex set $H$ is a complete graph on $pq-p$ vertices. Let $H'=  \langle a \rangle -\{e\}$. Then the induced subgraph of $\Delta^c(\mathbb Z_p \rtimes \mathbb Z_q)$ with vertex set $H'$ is a complete graph on $p-1$ vertices and no vertex of $H'$ is adjacent with any vertex of $H$. Sine $e$ is adjacent with every vertex , we have $\Delta^c(\mathbb Z_p \rtimes \mathbb Z_q)\cong K_{1,2}[K_1, K_{p-1}, K_{pq-p}].$
\end{proof}
We have  $\Delta^c(\mathbb Z_p \rtimes \mathbb Z_q)= \Delta^o(\mathbb Z_p \rtimes \mathbb Z_q)$ and both the adjacency and Laplacian spectrum of $\Delta^o(\mathbb Z_p \rtimes \mathbb Z_q)$  are already discussed in Theorem \ref{thm:adj} and Theorem \ref{thm:lap}, respectively.
\vskip 0.5cm
\noindent{\bf Conflict of interest:} On behalf of all authors, the corresponding author states that there is no conflict of interest.

\vskip .5cm

\noindent{\bf Addresses}: Sandeep Dalal$^*$, Sanjay Mukherjee$^{\dagger}$, Kamal Lochan Patra$^{\dagger}$\\

\noindent $^*$Department of Mathematics\\
 Punjab Engineering College (Deemed to be University)\\
 Chandigarh-160012, India.\\

\noindent $^{\dagger}$School of Mathematical Sciences\\
National Institute of Science Education and Research Bhubaneswar\\
An OCC of Homi Bhabha National Institute\\
Jatni, Khurda, Odisha, 752050, India.\\

\noindent{\bf E-mails}: deepdalal10@gmail.com, sanjay.mukherjee@niser.ac.in, klpatra@niser.ac.in
\end{document}